\documentclass[12pt]{amsart}

\usepackage{amsmath,amscd,amsthm,amssymb,enumitem,tikz-cd}

 \usepackage{xcolor}

 \usepackage{url}
 \usepackage{graphicx} %Para las figuras

\allowdisplaybreaks %align may cross pages

\newtheorem{thm}{Theorem}

\newtheorem{pro}[thm]{Proposition}
\newtheorem{cor}[thm]{Corollary}

\theoremstyle{definition}
\newtheorem{rmk}{Remark}
\newtheorem{defi}[rmk]{Definition}
\newtheorem{exa}[rmk]{Example}

\begin{document}

\title{Set of evolution operators of an evolution algebra}

\author{D.~Fern\'andez-Ternero \and V.M.~G\'omez-Sousa \and J.~N\'u\~nez-Vald\'es}

\thanks{The first and the second authors were partially supported by MICINN Spain Research Project
PID2020-118753GB-I00. All authors were partially suported by Junta de Andalucía Research Group FQM–326.}

\date{\today}

\begin{abstract}
An automorphism defined on an evolution algebra can provide both a finite number and an infinite number of evolution operators on it. This question is dealt with in the paper, as well as others more related to the evolution operators of evolution algebras. After defining the concept of the set of evolution operators of an evolution algebra, differences between trivial and non-trivial sets of evolution operators are also covered. Some properties of these concepts are studied and several examples of the above issues are shown.
\end{abstract}



\maketitle

%\tableofcontents

\section{Introduction}
Evolution algebras were firstly introduced in the Ph.D. Thesis of Jianjun Paul Tian in 2004 (\cite{phdthesis}). These algebras are non-associative and commutative, but unlike many other non-associative algebras such as Lie, Malcev or Jordan algebras, evolution algebras are not described by a series of identities. An \textit{evolution algebra} $E\equiv (E,+,\cdot)$ over a field $\mathbb{K}$ is a $\mathbb{K}$-algebra such that there exists a \textit{natural basis}, that is, a basis $\mathcal{B}=\{e_i:i=1,\dots,n\}$, such that $e_i\cdot e_j=0$, for all $i\neq j$. The scalars $a_{ij}$ such that $e_j^2=e_j\cdot e_j=\sum_{i=1}^na_{ij}e_i$ will be called the \textit{structure constants}, and the matrix $A=(a_{ij})$ is said to be the \textit{structure matrix}. Furthermore, the evolution algebra is called {\it non-degenerate} if $e_j^2\neq0$, for all $j$. Otherwise, the algebra is called {\it degenerate}. Therefore, all the dynamic information of an evolution algebra is collected in this matrix, from which the operator that is the object of study in this paper is defined.

Fixed a natural basis, the linear map $L:E\rightarrow E$ which maps each generator into its square, $L(e_j)=e_j^2=\sum_{i=1}^na_{ij}e_i$, is called the \textit{evolution operator} associated with $\mathcal{B}$. It is obvious that $M_\mathcal{B}(L)=A$, that is, the matrix representation of this linear transformation is the structure matrix.

Topics related to the evolution operator of evolution algebras are scant in the literature to date, being \cite{dvj1,dvj2,dvj3,ip,rv} some of the several papers that deal with this subject. In this work we delve into different aspects related to this operator in order to solve this shortage and provide new properties that make easier the study of this field.

Since the evolution operator depends on the chosen natural basis, it is not unique. Therefore, the following questions naturally arise
\begin{itemize}
    \item Do evolution operators share any properties?
    \item When do two evolution operators represent the same linear map?
    \item When writing all evolution operators in the same basis, how are they related?
    \item Is there any way to relate evolution operators and automorphisms?
\end{itemize}

The aim of this paper is to answer these questions in a general way. To do so, in Section 2 we briefly introduce the main definitions and properties that will be used, to later study, in Section 3, the connection between automorphisms and evolution operators. Finally, in Section 4 we study the structure of the set of all evolution operators in order to find connections between these linear maps.

\section{Preliminaries}

Below we define some concepts that will be used repeatedly throughout the paper.

\begin{defi}
Let $A=(a_{ij})$ and $B=(b_{ij})$ be two $n\times m$ matrices. We define the following matrix operations
\begin{itemize}
    \item The \textup{Hadamard product} $A\odot B=(a_{ij}b_{ij})$.
    \item The \textup{Hadamard power} $A^{(n)}=\underbrace{A\odot\dots\odot A}_{n\text{ times}}$.
    \item $A*B=(c_{ij}^k)$ a $n\times\frac{m(m-1)}{2}$ matrix whose  columns are indexed by pairs $(i,j)$ with $i<j$, considering the lexicographical order. The column $(i,j)$ is $A_{*i}\odot B_{*j}$, that is, the Hadamard product of the $i$th column of $A$ and the $j$th column of $B$.
\end{itemize}
\end{defi}

The following result, whose proof can be found at \cite{csv}, allows us to relate the product of the algebra to the Hadamard product and the evolution operator.

\begin{pro}
\label{prop}
Let $E$ be an evolution algebra with natural basis $\mathcal{B}$ and structure matrix $A$. Then, the following diagram is commutative
\[\begin{tikzcd}
E\times E \arrow{r}{\odot} \arrow[swap]{dr}{\cdot} & E \arrow{d}{L} \\
 & E
\end{tikzcd}\]
where $\odot:E\times E\rightarrow E$ is defined as
\[x\odot y=\left(\sum_{j=1}^nx_je_j\right)\odot\left(\sum_{j=1}^ny_je_j\right)=\sum_{j=1}^nx_jy_j\,e_j.\]
Equivalently, in matrix form, we get
\[\left(x\cdot y\right)_\mathcal{B}=A\left(x_\mathcal{B}\odot y_\mathcal{B}\right).\]
\end{pro}

Recall that, given two algebras $A$ and $A'$, a linear map $f:A\rightarrow A'$ is said to be an \textit{homomorphism of algebras} (\textit{homomorphism} for short) if $f(x\cdot y)=f(x)\cdot f(y)$, for every $x,y\in A$. A homomorphism from an algebra to itself is an \textit{endomorphism} and an \textit{automorphism} is a biyective endomorphism.

\vspace{0.15cm}

The following characterization of endomorphisms is well know. We take advantage to give a shorter alternative proof to the one shown in \cite{phdthesis}.

\begin{pro}
Let $E$ be an evolution algebra with natural basis $\mathcal{B}$ and structure matrix $A$. Then, a linear map $g:E\rightarrow E$ is an homomorphism if and only if $A(G*G)=0$ and $AG^{(2)}=GA$, where $G=M_\mathcal{B}(g)$.
\end{pro}
\begin{proof}
 A linear map $g$ is an homomorphism if and only if $g(e_i)\cdot g(e_j)=g(e_i\cdot e_j)=g(0)=0$ and $g(e_j)^2=g(e_j^2)$. Taking coordinates with respect to $\mathcal{B}$ and making use of Proposition \ref{prop}, the first equality is
\[0=\left(g(e_i)\cdot g(e_j)\right)_\mathcal{B}=A\left(g(e_i)_\mathcal{B}\odot g(e_j)_\mathcal{B}\right)=A\left(G_{*i}\odot G_{*j}\right),\]
so it is equivalent to $A(G*G)=0$. Since
\[\left(g(e_j)^2\right)_\mathcal{B}=Ag(e_j)_\mathcal{B}^{(2)}=AG_{*j}^{(2)},\]
\[\left(g(e_j^2)\right)_\mathcal{B}=G\left(e_j^2\right)_\mathcal{B}=GA_{*j},\]
the second equality is equivalent to $AG^{(2)}=GA$.   
\end{proof}

We end these preliminaries with the following results given in \cite{bcs}.

\begin{thm}
\label{teobcs}
Let $E$ be an evolution algebra with natural basis $\mathcal{B}=\{e_i: i=1,\dots,n\}$ and structure matrix $A=(a_{ij})$. Let $\mathcal{B}'=\{\eta_j=\sum_{i=1}^n g_{ij} e_i:j=1,\dots,n\}$ be another basis of $E$, with $supp(\eta_j)=\{i: g_{ij}\neq0\}=\{i_1,\dots,i_r\}$.
\begin{itemize}
    \item If $\eta_j^2=0$, then $e_{i_1}^2=\dots=e_{i_r}^2=0$.
    \item If $\eta_j^2\neq0$, then $rank(\{e_{i_1}^2,\dots,e_{i_r}^2\})=1$.
\end{itemize}
\end{thm}

\begin{defi}
Let $E$ be an evolution algebra with natural basis $\mathcal{B}$.
\begin{itemize}
    \item We say that $E$ {\it has a unique natural basis} if any other natural basis of $E$ is equal to $B$ up to rearrangement and product by nonzero scalars.
    \item We say that $E$ has {\it Property (2LI)} if for any two different vectors $e_i,e_j$ of $\mathcal{B}$, $\{e_i^2,e_j^2\}$ is linearly independent.
\end{itemize}
\end{defi}

\begin{cor}
\label{corbcs}
Let $E$ be a non-degenerate evolution algebra. Then the following assertions are equivalent
\begin{enumerate}
    \item $E$ has a unique natural basis.
    \item $E$ has Property (2LI).
\end{enumerate}
Degenerate evolution algebras do not satisfy any of these assertions.
\end{cor}

\section{Automorphisms and evolution operators}

\begin{pro}
\label{prop31}
Let $g:E\rightarrow E$ be an automorphism of an evolution algebra $E$ with natural basis $\mathcal{B}=\{e_i:i=1,\dots,n\}$ and structure matrix $A=(a_{ij})$. Then, $\mathcal{B'}=\{g(e_i):i=1,\dots,n\}$ is a natural basis of $E$ and the evolution operator $L'$ associated with this basis fulfills that
\[M_{\mathcal{B'}}(L')=A=M_{\mathcal{B}}(L).\]
Furthermore, $M_{\mathcal{B}}(L')=GAG^{-1}$ and these two operators are the same linear map if and only if $AG=GA$, where $G=M_{\mathcal{B}}(g)$.
\end{pro}

\begin{proof}
If $i\neq j$, then $g(e_i)\cdot g(e_j)=g(e_i\cdot e_j)=g(0)=0$, so $\mathcal{B'}$ is a natural basis.\\
Secondly, the evolution operator $L'$ maps each element of $\mathcal{B'}$ onto its square
\[g(e_i)^2=g(e_i^2)=g\left(\sum_{k=1}^n a_{ki}e_k\right)=\sum_{k=1}^n a_{ki}g(e_k),\]
that is to say, $[g(e_i)^2]_{\mathcal{B'}}=(a_{1i},\dots,a_{ni})^t$, which is the $i$-th column of $A$. Hence it follows that $M_{\mathcal{B'}}(L')=A$.\\
Finally, by the definition of $\mathcal{B'}$, we have that $G=M(\mathcal{B'},\mathcal{B})$, so $M_{\mathcal{B}}(L')=M(\mathcal{B'},\mathcal{B})M_{\mathcal{B'}}(L')M(\mathcal{B},\mathcal{B'})=GAG^{-1}$. This operator is the same as $L$ if and only if $A=GAG^{-1}$, or what is the same, $AG=GA$.
\end{proof}

\begin{rmk}
On the same assumptions as above, $G^nAG^{-n}$ is the matrix of an evolution operator with respect to the basis $\mathcal{B}$, for all $n\in\mathbb{Z}$, since $g^n$ is also an automorphism.
\end{rmk}

We show below two examples in which $G^nAG^{-n}$ provides a finite number of evolution operators and an infinite number of them, respectively.

\begin{exa}
Let $E$ be the evolution algebra with natural basis $\mathcal{B}=\{e_1,e_2,e_3\}$ and structure matrix
\[A=\begin{pmatrix}
a & a & b\\ 
(\sqrt{2}-1)a & (\sqrt{2}-1)a & (\sqrt{2}-1)b\\ 
c & c & d
\end{pmatrix}\!\!,\]
for some $a,b,c,d\in\mathbb{C}$. Let us consider the invertible linear map
\[G=\begin{pmatrix}
\frac{1}{\sqrt{2}} & \frac{1}{\sqrt{2}} & 0\\ 
\frac{1}{\sqrt{2}} & -\frac{1}{\sqrt{2}} & 0\\ 
0 & 0 & 1
\end{pmatrix}\!\!.\]
which satisfies that $G^{-1}=G$. This map is an automorphism, since
\[A(G*G)=A\begin{pmatrix}
\frac{1}{2} & 0 & 0\\ 
-\frac{1}{2} & 0 & 0\\ 
0 & 0 & 0
\end{pmatrix}=0 \textup{ and }AG^{(2)}=A=GA.\]
Then, $\mathcal{B'}=\left\{\frac{1}{\sqrt{2}}(e_1+e_2),\frac{1}{\sqrt{2}}(e_1-e_2),e_3\right\}$ is a natural basis whose evolution operator satisfies $M_{\mathcal{B'}}(L')=A$. The expression of this operator with respect to the basis $\mathcal{B}$ is
\[M_{\mathcal{B}}(L')=GAG^{-1}=GAG=AG=\begin{pmatrix}
\sqrt{2}a & 0 & b\\ 
(2-\sqrt{2})a & 0 & (\sqrt{2}-1)b\\ 
\sqrt{2}c & 0 & d
\end{pmatrix}\!\!,\]
which is different from $A$ if and only if $a$ or $c$ are nonzero. With this assumption, $L$ and $L'$ are not the same linear map, and there are two different evolution operators. Note that, since $G^2=G$, the matrix $G^nAG^{-n}$ is $A$ or $AG$.
\end{exa}

The following example shows that there can be infinite evolution operators, different from each other, but with the same matrix with respect to its natural basis.

\begin{exa}
Let $E$ be the evolution algebra with natural basis $\mathcal{B}=\{e_1,e_2,e_3\}$ and structure matrix
\[A=\begin{pmatrix}
1 & 1 & 1\\ 
-\frac{i}{2} & -\frac{i}{2} & -\frac{i}{2}\\ 
-\frac{i}{2} & -\frac{i}{2} & -\frac{i}{2}
\end{pmatrix}\!\!,\]
Let us consider the invertible linear map
\[G=\begin{pmatrix}
\frac{5}{3} & -\frac{4i}{3} & 0\\ 
0 & 0 & 1\\ 
-\frac{4i}{3} & -\frac{5}{3} & 0
\end{pmatrix}\!\!,\]
which satisfies that $G^{-1}=G^t$. This map is an automorphism, since
\[A(G*G)=A\begin{pmatrix}
-\frac{20i}{9} & 0 & 0\\ 
0 & 0 & 0\\ 
\frac{20i}{9} & 0 & 0
\end{pmatrix}=0 \textup{ and }AG^{(2)}=A=GA.\]
Then, $\mathcal{B'}=\left\{\frac{1}{3}(5e_1-4ie_3),-\frac{1}{3}(4ie_1+5e_3),e_2\right\}$ is a natural basis whose evolution operator satisfies $M_{\mathcal{B'}}(L')=A$. The expression of this operator with respect to the basis $\mathcal{B}$ is
\[M_{\mathcal{B}}(L')=GAG^{-1}=GAG^t=AG^t=\begin{pmatrix}
\frac{5-4i}{3} & 1 & -\frac{5+4i}{3}\\ 
-\frac{4+5i}{6} & -\frac{i}{2} & \frac{-4+5i}{6}\\ 
-\frac{4+5i}{6} & -\frac{i}{2} & \frac{-4+5i}{6}
\end{pmatrix}\!\!,\]
which is different from $A$, so $L$ and $L'$ are not the same linear map. Let us see that the matrix $G^nAG^{-n}$, for $n\geq0$, provides infinite evolution operators, all of them different from each other.\\
Since $GA=A$, then $G^nAG^{-n}=AG^{-n}=A(G^t)^n$. By contradiction, suppose there exist $m>n\geq0$, such that $A(G^t)^m=A(G^t)^n$, or what is the same
\begin{align} 
A((G^t)^k-I)=0,\label{eq1}
\end{align}
where $k=m-n$. The eigenvalues of $G^t$ are $1$, $\lambda$ and $\frac{1}{\lambda}$, where
\[\lambda=\frac{1}{3}+\frac{2\sqrt{2}}{3}\,i.\]
Since $G^t$ has three different eigenvalues, it is diagonalizable, so there exists $P$ such that $G^t=PDP^{-1}$, with $D=diag(1,\lambda,\frac{1}{\lambda})$. Then, (\ref{eq1}) is equivalent to $AP(D^k-I)P^{-1}=0$, and thus, $AP(D^k-I)=0$. Let us denote the $i$-th column of $P$ as $p_i$. Then we have
\[0=A\begin{pmatrix}
    \vrule & \vrule & \vrule\\
    p_1   & p_2   & p_3\\
    \vrule & \vrule & \vrule
\end{pmatrix}
\begin{pmatrix}
0 & 0 & 0\\ 
0 & \lambda^k-1 & 0\\ 
0 & 0 & \frac{1}{\lambda^k}-1
\end{pmatrix}=$$ $$=A\begin{pmatrix}
    0 & \vrule & \vrule\\
    0 & (\lambda^k-1)p_2   & (\frac{1}{\lambda^k}-1)p_3\\
    0 & \vrule & \vrule
\end{pmatrix}\!\!.\]
We distinguish two cases
\begin{itemize}
    \item If $\lambda^k\neq1$, then $p_2,p_3\in null(A)$. Therefore, $0=AP(D-I)P^{-1}=A(G^{t}-I)$, which is a contradiction due to $AG^t\neq A$.
    \item If $\lambda^k=1$, since the argument of $\lambda$ is $\alpha=\arccos{\left(\frac{1}{3}\right)}$, then $k\alpha=2\pi l$, with $l\in\mathbb{Z}$. This is equivalent to $\frac{\alpha}{\pi}=\frac{2l}{k}$, which is a contradiction because $\frac{\alpha}{\pi}$ is not a rational number.
\end{itemize}
\end{exa}

In the previous example we have made use of the following proposition, which was given by Aigner and Ziegler in \cite{AZ}.

\begin{pro}
For every odd integer $n\geq3$, the number $\frac{1}{\pi}\arccos{\left(\frac{1}{\sqrt{n}}\right)}$ is irrational.
\end{pro}

As a consequence of Proposition \ref{prop31} we get the following result.

\begin{cor}
Let $E$ be an evolution algebra with natural basis $\mathcal{B}=\{e_i:i=1,\dots,n\}$ satisfying Property (2LI). Then $Aut(E)\subseteq S_n\rtimes(\mathbb{K}^\times)^n$.
\end{cor}
\begin{proof}
Let $g$ be an automorphism. Then, $\mathcal{B}'=\{g(e_i):i=1,\dots,n\}$ is a natural basis by Proposition \ref{prop31}. Since $E$ has Property (2LI), then $E$ has a unique natural basis by Corollary \ref{corbcs}. That is, $\mathcal{B}'$ is obtained from $\mathcal{B}$ by a rearrangement and product by nonzero scalars.
\end{proof}

\begin{pro}
\label{prop36}
Let $E$ be a real evolution algebra with natural basis $\mathcal{B}=\{e_i:i=1,\dots,n\}$ and structure matrix $A=(a_{ij})$. Let suppose there exists $\pi\in S_n$ such that $a_{i\pi(i)}\neq0$, for all $i=1,\dots,n$. Then, $Aut(E)\cap (\mathbb{R}^\times)^n=\{id\}$.
\end{pro}
\begin{proof}
A linear map $g$ satisfies $g\in Aut(E)\cap (\mathbb{R}^\times)^n$ if and only if $G=M_\mathcal{B}(g)=Diag(\lambda)$, for some $\lambda=(\lambda_i)_{i=1}^n$, $\lambda_i\neq0$, $A(G*G)=0$ and $AG^{(2)}=GA$. The identity $A(G*G)=0$ holds trivially and $AG^{(2)}=GA$ is equivalent to $a_{ij}\lambda_j^2=\lambda_i a_{ij}$, that is $a_{ij}=0$ or $\lambda_i=\lambda_j^2$.

Le us write $\pi$ as a product of disjoint cycles, $\pi=c_1\circ\dots\circ c_r$. For all $i$ there is a unique $j$ such that $i\in supp(c_j)$. Let $l_j$ be the length of the cycle $c_j$. Then,
\[\lambda_i=\lambda_{c_j(i)}^2=\dots=\lambda_{c_j^{l_j}(i)}^{2^{l_j}}=\lambda_i^{2^{l_j}}.\]
The only real non-zero solution of the equation $\lambda_i=\lambda_i^{2^{l_j}}$ is $\lambda_i=1$, so $g=id$.
\end{proof}

\begin{pro}
Let $E$ be a real evolution algebra with natural basis $\mathcal{B}=\{e_i:i\in\Lambda=\{1,\dots,n\}\}$ and structure matrix $A=(a_{ij})$ without null rows. Then, $Aut(E)\cap (\mathbb{R}^\times)^n=\{id\}$.
\end{pro}
\begin{proof}
We proceed by induction in $n$. For $n=1$, it is trivial. We assume that the statement holds for dimension $n-1$ in order to prove that it holds for dimension $n$.

If $A=Diag(\lambda)P$, with $\lambda=(\lambda_i)_{i\in\Lambda}$, $\lambda_i\neq0$ and $P$ being a permutation matrix, then the result follows from Proposition \ref{prop36}.

In another case, as in the proof of Proposition \ref{prop36}, we have $g\in Aut(E)\cap (\mathbb{R}^\times)^n$ if and only if $G=M_\mathcal{B}(g)=Diag(\lambda)$ and either $a_{ij}=0$ or $\lambda_i=\lambda_j^2$, for all $i,j\in\Lambda$. Since $A\neq Diag(\lambda)P$, there exists $k\in\Lambda$ such that the matrix $A'=(a_{ij})_{i,j\in\Lambda\setminus\{k\}}$ has no null rows. By induction hypothesis, $\lambda_i=1$, for all $i\in\Lambda\setminus\{k\}$. Since row $k$ is not null, there exists $l$ such that $a_{kl}\neq0$ and then $\lambda_k=\lambda_l^2$. If $k\neq l$, then $\lambda_k=\lambda_l^2=1^2=1$. If $k=l$, then $\lambda_k=\lambda_k^2$ and the only non-zero solution of this equation is $\lambda_k=1$. Therefore, $\lambda_i=1$, for all $i\in\Lambda$, so $g=id$.
\end{proof}

\begin{exa}
Let $E$ be the complex evolution algebra with natural basis $\mathcal{B}=\{e_1,e_2\}$ and structure matrix
\[A=\begin{pmatrix}
0 & a\\ 
b & 0
\end{pmatrix}\!\!,\]
with $a,b\neq0$. Let us consider the linear map $g:E\rightarrow E$ with $G=M_\mathcal{B}(g)=Diag(\lambda_1,\lambda_2)$ and $\lambda_1=-\frac{1}{2}+\frac{\sqrt{3}}{2}i$, $\lambda_2=-\frac{1}{2}-\frac{\sqrt{3}}{2}i$. Then, $g\in Aut(E)\cap (\mathbb{C}^\times)^n$.
\end{exa}

\begin{exa}
Let $E$ be the real evolution algebra with natural basis $\mathcal{B}=\{e_1,e_2\}$ and structure matrix
\[A=\begin{pmatrix}
a & b\\ 
0 & 0
\end{pmatrix}\!\!,\]
with $a,b\neq0$. Let us consider the linear map $g:E\rightarrow E$ with $G=M_\mathcal{B}(g)=Diag(1,-1)$. Then, $g\in Aut(E)\cap (\mathbb{R}^\times)^n$.
\end{exa}

\begin{pro}
Let $E$ be an evolution algebra with natural basis $\mathcal{B}=\{e_i:i=1,\dots,n\}$ and structure matrix $A=(a_{ij})$. Then, $S_n\subseteq Aut(E)$ if and only if $A=\alpha J + \beta I$, where $J$ is the all-ones matrix and $\alpha,\beta\in\mathbb{K}$.
\end{pro}
\begin{proof}
Let $P$ be a permutation matrix. This permutation is an automorphism if and only if $A(P*P)=0$ and $AP^{(2)}=PA$. The first equality is trivially satisfied and the second is equivalent to $AP=PA$. The only matrices that commute with any permutation matrix are those of the form $\alpha J + \beta I$.
\end{proof}

\section{Set of evolution operators}

\begin{pro}
Let $E$ be an evolution algebra with natural basis $\mathcal{B}=\{e_i:i=1,\dots,n\}$ and structure matrix $A=(a_{ij})$. Let $\mathcal{B'}=\{\eta_i:i=1,\dots,n\}$ be another basis of $E$ and $G=M_{\mathcal{B}}(g)$, where $g$ is the linear map $g(e_i)=\eta_i$. Then, $\mathcal{B'}$ is a natural basis if and only if $A(G*G)=0$. In this case, the evolution operator $L'$ associated with this basis fulfills that
\[M_{\mathcal{B'}}(L')=G^{-1}AG^{(2)}, M_{\mathcal{B}}(L')=AG^{(2)}G^{-1}.\]
Furthermore, these two operators are the same linear map if and only if $AG=AG^{(2)}$.
\end{pro}

\begin{proof}
The basis $\mathcal{B'}$ is a natural basis if and only if $\eta_i\cdot\eta_j=0$, for all $i\neq j$. By the definition of $g$, this is the same as $g(e_i)\cdot g(e_j)=0$. Taking coordinates with respect to $\mathcal{B}$,
\[0=\left(g(e_i)\cdot g(e_j)\right)_\mathcal{B}=A\left(g(e_i)_\mathcal{B}\odot g(e_j)_\mathcal{B}\right)=A\left(G_{*i}\odot G_{*j}\right),\]
so this is equivalent to $A(G*G)=0$. Since for all $j=1,\dots,n$
\[\left(\eta_j^2\right)_\mathcal{B'}=M(\mathcal{B},\mathcal{B'})\left(\eta_j^2\right)_\mathcal{B}=G^{-1}A\left(\eta_j\right)_\mathcal{B}^{(2)}=G^{-1}AG_{*j}^{(2)},\]
then $M_{\mathcal{B'}}(L')=G^{-1}AG^{(2)}$. Finally, \[M_{\mathcal{B}}(L')=M(\mathcal{B'},\mathcal{B})M_{\mathcal{B'}}(L')M(\mathcal{B},\mathcal{B'})=GG^{-1}AG^{(2)}G^{-1}=AG^{(2)}G^{-1}.\] This operator is the same as $L$ if and only if $A=AG^{(2)}G^{-1}$, or what is the same, $AG=AG^{(2)}$.
\end{proof}

\begin{defi}
Given an evolution algebra $E$ with natural basis $\mathcal{B}$ and structure matrix $A$, we define the set of all evolution operators with respect to its own natural basis as
\[\mathcal{L}=\left\{G^{-1}AG^{(2)}: A(G*G)=0, rank(G)=n\right\}\]
and with respect to the natural basis $\mathcal{B}$ as
\[\mathcal{L}_{\mathcal{B}}=\left\{AG^{(2)}G^{-1}: A(G*G)=0, rank(G)=n\right\}.\]
\end{defi}

\begin{pro}
All evolution operators of an evolution algebra $E$ have the same rank.
\end{pro}

\begin{proof}
Let $\mathcal{B},\mathcal{B}'$ be two natural basis and $L,L'$ the evolution operators associated with these basis, respectively. Let $A:=M_{\mathcal{B}}(L)$ and $A':=M_{\mathcal{B}'}(L')=G^{-1}AG^{(2)}$. Then,
\[rank(A')=rank(G^{-1}AG^{(2)})=rank(AG^{(2)})\leq rank(A).\]
Exchanging the roles of $A$ and $A'$ it follows that $rank(A')\leq rank(A)$. Therefore $rank(A)=rank(A')$.
\end{proof}

\begin{pro}
$ADiag(\lambda)\in\mathcal{L}_{\mathcal{B}}$ for all $\lambda=(\lambda_i)_{i=1}^n\in\mathbb{C}^n$ with $\lambda_i\neq0$, for all $i=1,\dots,n$.
\end{pro}

\begin{proof}
Let $G=Diag(\lambda)$. Then, $G*G$ is the zero matrix and $A(G*G)=0$, so $AG^{(2)}G^{-1}\in\mathcal{L}_\mathcal{B}$. That is to say,
\begin{align*}
    AG^{(2)}G^{-1}&=ADiag\left(\lambda\right)^{(2)}Diag\left(\lambda\right)^{-1}=\\
    &=ADiag\left(\lambda^{(2)}\right)Diag\left(\lambda^{(-1)}\right)=\\
    &=ADiag\left(\lambda^{(2)}\lambda^{(-1)}\right)=ADiag(\lambda)\in\mathcal{L}_\mathcal{B}.
\end{align*}
\end{proof}

\begin{defi}
$\mathcal{L}_{\mathcal{B}}$ is said to be \textup{trivial} if
\[\mathcal{L}_{\mathcal{B}}=\left\{ADiag(\lambda): \lambda=(\lambda_i)_{i=1}^n\in\mathbb{C}^n, \lambda_i\neq0\right\},\]
and \textup{semitrivial} if
\[\mathcal{L}_{\mathcal{B}}\subseteq\left\{ADiag(\lambda): \lambda=(\lambda_i)_{i=1}^n\in\mathbb{C}^n\right\}.\]
\end{defi}

\begin{pro}
Let $E$ be an evolution algebra with natural basis $\mathcal{B}=\{e_i:i=1,\dots,n\}$ and structure matrix $A=(a_{ij})$ satisfying Property (2LI). Then, $\mathcal{L}_{\mathcal{B}}$ is trivial.
\end{pro}

\begin{proof}
If $E$ has Property (2LI) then $E$ has a unique natural basis by Corollary \ref{corbcs}, that is $G=Diag(\lambda)P$, with $\lambda_i\neq0$ for all $i=1,\dots,n$ and $P$ a permutation matrix. Then,
\begin{align*}
    AG^{(2)}G^{-1}&=A(Diag(\lambda)P)^{(2)}(Diag(\lambda)P)^{-1}=\\
    &=ADiag(\lambda)^{(2)}PP^{-1}Diag(\lambda)^{-1}=\\
    &=ADiag(\lambda^{(2)})Diag(\lambda^{(-1)})=ADiag(\lambda).
\end{align*}
\end{proof}

The following examples show that if $E$ does not satisfy (2LI) then $L_\mathcal{B}$ may or may not be trivial.

\begin{exa}
Let $E$ be the evolution algebra with natural basis $\mathcal{B}=\{e_1,e_2\}$ and structure matrix
\[A=\begin{pmatrix}
a & 0\\ 
b & 0
\end{pmatrix}\!\!,\]
with $a\neq0$ or $b\neq0$. Then, $null(A)=\{x_1=0\}$ and $G=\begin{pmatrix}
g_1 & g_2\\ 
g_3 & g_4
\end{pmatrix}$ satisfies $A(G*G)=0$ if and only if $g_1g_2=0$. It is easy to see that

\begin{itemize}
    \item If $g_1=0$, then
\begin{align*}
AG^{(2)}G^{-1}&=\begin{pmatrix}
a & 0\\ 
b & 0
\end{pmatrix}\begin{pmatrix}
0 & g_2^2\\ 
g_3^2 & g_4^2
\end{pmatrix}\frac{1}{det(G)}\begin{pmatrix}
g_4 & -g_2\\ 
-g_3 & 0
\end{pmatrix}=\\[5pt]
&=\frac{1}{-g_2g_3}\begin{pmatrix}
-ag_2^2g_3 & 0\\ 
-bg_2^2g_3 & 0
\end{pmatrix}=\begin{pmatrix}
ag_2 & 0\\ 
bg_2 & 0
\end{pmatrix}\!\!.
\end{align*}
    \item If $g_2=0$, then
\begin{align*}
AG^{(2)}G^{-1}&=\begin{pmatrix}
a & 0\\ 
b & 0
\end{pmatrix}\begin{pmatrix}
g_1^2 & 0\\ 
g_3^2 & g_4^2
\end{pmatrix}\frac{1}{det(G)}\begin{pmatrix}
g_4 & 0\\ 
-g_3 & g_1
\end{pmatrix}=\\[5pt]
&=\frac{1}{g_1g_4}\begin{pmatrix}
ag_1^2g_4 & 0\\ 
bg_1^2g_4 & 0
\end{pmatrix}=\begin{pmatrix}
ag_4 & 0\\ 
bg_4 & 0
\end{pmatrix}\!\!.
\end{align*}
\end{itemize}
Therefore $\mathcal{L}_{\mathcal{B}}$ is trivial.
\end{exa}

\begin{exa}
Let $E$ be the evolution algebra with natural basis $\mathcal{B}=\{e_1,e_2\}$ and structure matrix
\[A=\begin{pmatrix}
a & a\\ 
b & b
\end{pmatrix}\!\!,\]
with $a\neq0$ or $b\neq0$. Then, $G=\begin{pmatrix}
1 & -1\\ 
1 & 1
\end{pmatrix}$ satisfies $A(G*G)=0$ and
\[AG^{(2)}G^{-1}=\begin{pmatrix}
a & a\\ 
b & b
\end{pmatrix}\begin{pmatrix}
1 & 1\\ 
1 & 1
\end{pmatrix}\frac{1}{det(G)}\begin{pmatrix}
1 & 1\\ 
-1 & 1
\end{pmatrix}=\begin{pmatrix}
0 & 2a\\ 
0 & 2b
\end{pmatrix}\!\!.\]
Note that this matrix is not of the form $ADiag(\lambda_1,\lambda_2)$ with $\lambda_i\neq0$, so $\mathcal{L}_\mathcal{B}$ is not trivial. Nevertheless, $AG^{(2)}G^{-1}=ADiag(0,2)$, so $\mathcal{L}_\mathcal{B}$ could be semitrivial.
\end{exa}

\begin{pro}
Let $E$ be an evolution algebra with natural basis $\mathcal{B}=\{e_i:i=1,\dots,n\}$. Then, $\mathcal{L}_{\mathcal{B}}$ is semitrivial.
\end{pro}
\begin{proof}
Let $\mathcal{B}'=\{\eta_i:i=1,\dots,n\}$ be another natural basis and let us write $e_j=\sum_{k=1}^{r_j} g_{i_kj}'\eta_{i_k}$, with $g_{i_kj}'\neq0$. We distinguish two cases
\begin{itemize}
    \item If $e_j^2=0$, then we deduce that $\eta_{i_k}^2=0$ by Theorem \ref{teobcs}. Therefore
    \begin{align*}
        L'(e_j)&=L'\left(\sum_{k=1}^{r_j} g_{i_kj}'\eta_{i_k}\right)=\sum_{k=1}^{r_j} g_{i_kj}'L'(\eta_{i_k})=\sum_{k=1}^{r_j} g_{i_kj}'\eta_{i_k}^2=0,\\
        L(e_j)&=e_j^2=0.
    \end{align*}
    Hence $L'(e_j)=L(e_j)$.
    \item If $e_j^2\neq0$, then we deduce that $\eta_{i_k}^2=\alpha_{i_k}v$ for some $v\in E$ by Theorem \ref{teobcs}. Therefore
    \begin{align*}
        L'(e_j)&=\sum_{k=1}^{r_j} g_{i_kj}'\eta_{i_k}^2=\left(\sum_{k=1}^{r_j} g_{i_kj}'\alpha_{i_k}\right)v,\\
        L(e_j)&=e_j^2=\left(\sum_{k=1}^{r_j} g_{i_kj}'\eta_{i_k}\right)^2=\sum_{k=1}^{r_j}g_{i_kj}'^2\eta_{i_k}^2=\left(\sum_{k=1}^{r_j} g_{i_kj}'^2\alpha_{i_k}\right)v\neq 0.
    \end{align*}
    Hence $L'(e_j)=\lambda_j L(e_j)$ with $\displaystyle\lambda_j=\frac{\sum_{k=1}^{r_j} g_{i_kj}'\alpha_{i_k}}{\sum_{k=1}^{r_j} g_{i_kj}'^2\alpha_{i_k}}$.
\end{itemize}
Therefore, there exists $\lambda_j$ such that $L'(e_j)=\lambda_j L(e_j)$, for all $j=1,\dots,n$. From this we deduce that $L_\mathcal{B}$ is semitrivial, since $M_\mathcal{B}(L')=M_\mathcal{B}(L)Diag(\lambda)=ADiag(\lambda)$, with $\lambda=(\lambda_j)_{j=1}^n$.
\end{proof}

\begin{rmk}
In the previous proof, we can not claim that $E$ is trivial, since it is necessary that $\lambda_j\neq0$, that is, $\sum_{k=1}^{r_j} g_{i_kj}'\alpha_{i_k}\neq0$, for all $j=1,\dots,n$ with $e_j^2\neq0$.
\end{rmk}

\begin{pro}
Let $E$ be an evolution algebra with natural basis $\mathcal{B}=\{e_i: i=1,\dots,n\}$. Let suppose there exist $r,s$ and $\alpha\neq0$ such that $e_r^2=\alpha e_s^2$ and $e_s^2\neq 0$. Then, $\mathcal{L}_\mathcal{B}$ is not trivial.
\end{pro}
\begin{proof}
Let $\beta$ be such that $\beta^2=\alpha$ and consider the basis $\mathcal{B}'=\{\eta_i:i=1,\dots,n\}$ with
\[\eta_i=\begin{cases}
e_i & \text{ if } i\neq r,s, \\
e_r+\beta e_s & \text{ if } i=r, \\
e_r-\beta e_s & \text{ if } i=s.
\end{cases}\]
Then, $\mathcal{B}'$ is a natural basis due to
\[\eta_r\cdot\eta_s=(e_r+\beta e_s)\cdot(e_r-\beta e_s)=e_r^2-\beta^2e_s^2=e_r^2-\alpha e_s^2=0.\]
Furthermore, since $\eta_r^2=\eta_s^2=(1+\alpha)e_s^2$, we have that
\[L'(e_s)=L'\left(\frac{1}{2\beta}(\eta_r-\eta_s)\right)=\frac{1}{2\beta}(\eta_r^2-\eta_s^2)=0.\]
Since $e_s^2\neq0$, then $\mathcal{L}_\mathcal{B}$ is not trivial.
\end{proof}

\begin{cor}
Let $E$ be a non-degenerate evolution algebra with natural basis $\mathcal{B}=\{e_i: i=1,\dots,n\}$ and structure matrix $A$ not satisfying Property (2LI). Then, $\mathcal{L}_\mathcal{B}$ is not trivial.
\end{cor}

\begin{cor}
Let $E$ be a non-degenerate evolution algebra with natural basis $\mathcal{B}$. Then the following assertions are equivalent
\begin{enumerate}
    \item $E$ has a unique natural basis.
    \item $E$ has Property (2LI).
    \item $\mathcal{L}_\mathcal{B}$ is trivial.
\end{enumerate}
As a consequence, the fact of $\mathcal{L}_\mathcal{B}$ being trivial is an intrinsic property of the algebra, that is, it does not depend on the natural basis $\mathcal{B}$.
\end{cor}

\section{Conclusions}

In this work we have related evolution operators with homomorphisms, in addition to introducing new concepts such as the set of evolution operators of an evolution algebra, studying whether it is trivial or not. To the best of our knowledge, there are no other papers in the literature studying the relations between the different evolution operators, and our work can serve as a starting point for research on this topic, of which some questions remain open. For example, it can be studied if the fact of $\mathcal{L}_\mathcal{B}$ being trivial or semitrivial depends on the natural basis $\mathcal{B}$ in the degenerate case.

 \bibliographystyle{plain}

 \bigskip

\address{\noindent Dpto. de Geometr\'{\i}a y Topolog\'{\i}a, Universidad de Sevilla, C/ Tarfia S/N, 
            Sevilla, 41012, Spain.\\}
\email{desamfer@us.es, victor.manuel.gomez.sousa@gmail.com, jnvalde@us.es}

\end{document}